\documentclass[12pt]{amsart}
\usepackage{amssymb,amsmath,amsthm,fullpage,color,soul}
\usepackage{pdfsync}
\usepackage{tabu}

\newtheorem{thm}{Theorem}[section]

\newtheorem{cor}[thm]{Corollary}

\theoremstyle{definition}
\newtheorem{defn}[thm]{Definition}

\theoremstyle{remark}
\newtheorem{rem}[thm]{Remark}

\newcommand{\R}{\mathbb R}

\newcommand{\ZZ}{{\bar Z}}

\newcommand{\C}{\mathbb C}
\DeclareMathOperator{\supp}{supp}
\DeclareMathOperator{\Imm}{Im}

\DeclareMathOperator{\Rre}{Re}

\newcommand{\p}{\partial}
\newcommand{\les}{\lesssim}

\newcommand{\z}{\bar z}
\newcommand{\w}{\bar w}

\newcommand{\dbarb}{\bar\partial_b}

\newcommand{\lam}{\lambda}

\newcommand{\opF}{\mathcal F}

\newcommand{\mvs}{\varsigma}

\newcommand{\pure}[1]{\ensuremath{\frac{\partial}{\partial #1}}}

\begin{document}

\title{The Szeg\"o kernel on a class of noncompact CR manifolds of high codimension}

\author{Andrew Raich and Michael Tinker}

\thanks{This work was partially supported by a grant from the Simons Foundation (\#280164 to Andrew Raich).}
\thanks{The $n=2$ case in this paper was part of Tinker's Ph.D. thesis which he completed under Raich's supervision.}

\address{Department of Mathematical Sciences, SCEN 301, 1 University of Arkansas, Fayetteville, AR 72701}
\email{araich@uark.edu}

\address{Plano, TX}
\email{michael.tinker@ca.com}



\begin{abstract}We generalize Nagel's formula for the Szeg\"o kernel and use it to compute the Szeg\"o kernel on a class of noncompact CR manifolds
whose tangent space decomposes into one complex direction and several totally real directions. We also discuss the control metric on these
manifolds and relate it to the size of the Szeg\"o kernel.
\end{abstract}

\maketitle

\section{Introduction}
\label{sec:introduction}
The goal of this note is to derive a formula for the Szeg\"o kernel for a class of polynomial models that are CR manifolds whose maximal complex tangent space is one (complex)
dimensional and totally real tangent space is $n$ (real) dimensional. We also discuss the size of the Szeg\"o kernel in relation to the control metric.

When a CR manifold $M$ has a one (complex) dimensional maximal complex tangent space, then it is standard practice to identify $\dbarb$ with a vector field $\ZZ$ that is antiholomorphic and tangential.
The Szeg\"o kernel is then the orthogonal projection $S_M$ of $L^2(M)$ onto $L^2(M)\cap \ker \ZZ$. In complex analysis, the Szeg\"o kernel is a fundamental object of study, yet very little is known about the 
Szeg\"o kernel when the tangent space to $M$ has at least two totally real directions. In the case that $M$ is a quadric submanifold (with no hypothesis on the dimensionality of the maximal
complex tangent space), then researchers have computed the partial Fourier transform of the
$\Box_b$-heat kernel, from which the partial Fourier transform of the Szeg\"o kernel can be obtained \cite{BoRa11, CaChMa09, CaChFu11}. This article represents the first time that a formula/estimate
for the Szeg\"o kernel has been obtained for any example outside of quadrics when $M$ is not of hypersurface type.

In an interesting twist, we will see in Section \ref{sec:control geometry} that the control metric on $S_M$ is finite on the manifold where the Szeg\"o kernel in nonzero. This behavior may provide a clue
as to the behavior of the Szeg\"o kernel in higher codimensions when for every point, the span of the antiholomorphic vector fields is a strictly smaller dimension than the the dimension of the tangent space.

Let $p_1,\dots,p_n:\R\to\R$ be a collection of $n$ functions and $P = (p_1,\dots, p_n)$. The functions $p_j$ will typically be convex polynomials.
Our model $M_P$ will be a subset of $\C\times\C^n$, and we denote coordinates on
$\C\times\C^n$ by $(z,w)$ where $z = x+iy\in\C$ and $w = t+is \in \C^n$. Set
\[
\frac{\p}{\p\w} = \bigg( \frac{\p}{\p\w_1},\dots, \frac{\p}{\p\w_n}\bigg).
\]
Define
\begin{equation}\label{eqn:M defn}
M_P = \big\{ (z,w)\in \C\times\C^n : \Imm w = P(x) \}.
\end{equation}
The maximal complex tangent space is spanned by the vector
\[
\ZZ = \frac{\p}{\p\z} - iP'(x) \cdot \frac{\p}{\p\w}.
\]
Since the maximal complex tangent space has one dimension, the Szeg\"o kernel on $M$ is the orthogonal projection $S : L^2(M) \to L^2(M)\cap\ker \ZZ$.

We may identify $M$ with $\C\times\R^n$ under the identification
\[
\big(z,t+iP(x)\big) \longleftrightarrow (z,t).
\]
Under this identification, the vector field $2\ZZ$ pushes forward to the vector field
\[
\bar L = \frac{\p}{\p x} + i\Big( \frac{\p}{\p y} - P'(x)\cdot \frac{\p}{\p t}\Big).
\]
We have a choice of measure to put on 
$M$ (and consequently on $\C\times\R^n$). If $n=1$ and $P(x) = x^2$, then $M$ is the Heisenberg group $\mathbb H^1$ and Haar measure on $M$ corresponds to Lebesgue measure on $\C\times\R$.
Following precedent \cite{Nag86,Chr91,Rai06f,Rai06h,Rai07,Rai12,BoRa13h,BoRa11,BoRa09,NaSt06,NaRoStWa89,Has94,Str09h}, we use Lebesgue measure on
$\C\times\R^n$.

We can then identify the Szeg\"o projection $S_P$ on $L^2(M)$ with a projection that (by an abuse of notation) we also call the Szeg\"o projection and denote by $S_P$;
namely,  
the orthogonal projection of $L^2(\C\times\R^n)$ onto $L^2(\C\times\R^n)\cap \ker\bar L$. By standard Hilbert space theory, this Szeg\"o projection $S_P$ is given by
integration against a kernel $S_P\big((x,y,t),(x',y',t')\big)$, that is,
\[
S_Pf(x,y,t) = \int_{\C\times\R^n} S_P\big((x,y,t),(x',y',t')\big) f(x',y',t')\, dx'\, dy'\, dt'.
\]
The first goal of this paper is to find a tractable expression for $S_{P}\big((x,y,t),(x',y',t')\big)$.
\begin{thm}\label{thm:expression for S}
Let $M_P$ be a polynomial model defined by \eqref{eqn:M defn}. Then the Szeg\"o kernel for $M_P$ is given by the formula
\[
S_P\big((x,y,t),(x',y',t')\big)
= \int_{\Sigma_P} \frac{1}{C_{\eta,\tau}} e^{2\pi\eta((x+x')+i(y-y'))} e^{-2\pi\tau\cdot(P(x)+P(x')-i(t-t'))}\, d\eta\,d\tau
\]
where
\[
C_{\eta,\tau} = \int_\R e^{4\pi(x'\eta - P(x')\cdot\tau)}\, dx'
\]
and $\Sigma_P = \{(\eta,\tau)\in\R\times\R^n : C_{\eta,\tau}<\infty\}$.
\end{thm}

Theorem \ref{thm:expression for S} generalizes the Szeg\"o kernel formula of Nagel \cite[p.302]{Nag86}. In \cite{Nag86}, Nagel investigates the case
$M_p = \{(z,w) \in \C^2: \Imm w = p(w)\}$ where $p$ is a convex polynomial. If $C^1_{\eta,r} = \int_\R e^{4\pi(\eta x - rp(x))}\, dx$ and
$\Sigma_p = \{(\eta,r) : C^1_{\eta,r}<\infty\}$, then Nagel proves that
\begin{equation}\label{eqn:1-d Szego}
S_p\big((x,y,t),(x',y',t')\big) = \int_{\Sigma_p} \frac{1}{C^1_{\eta,\tau}} e^{2\pi\eta((x+x')+i(y-y'))} e^{-2\pi\tau\cdot(p(x)+p(x')-i(t-t'))}\, d\eta\,d\tau.
\end{equation}


We now explore several consequences of Theorem \ref{thm:expression for S}. 
\begin{thm}\label{thm:P = ap}
Let $M_P$ be a model defined by \eqref{eqn:M defn}, and assume that $p_j(x) = a_jp(x)$ for $1 \leq j \leq n$ where $a_n=1$ and $p(x)$ is a smooth function
satisfying $\lim_{|x|\to\infty}\frac{p(x)}{|x|} = \infty$. If we let $t = (s,t_n)$ and $a=(b,1)$, then
\[
S_P\big((x,y,t),(x',y',t')\big) = \delta_0[(s-s') - b(t_n-t_n')] S_p\big((x,y,t_n),(x',y',t_n')\big)
\]
where $\delta_0$ is the Dirac-$\delta$ in $\R^{n-1}$.
\end{thm}

The size of the Szeg\"o kernel when $M_p$ is a model of three real dimensions and $p$ is a convex polynomial is well understood 
\cite{Nag86,NaRoStWa89,Chr91,Rai06h}. In particular if $d(\cdot,\cdot)$ is the control metric
generated by the vector fields $X_1 = \Rre \bar L$ and $X_2 = \Imm \bar L$, and $B_{CC}(\alpha,\delta)$ is the control ball of radius $\delta$, then if $X^J$ is a multiindex of operators $X_1, X_2$ acting
in either $\alpha = (x,y,t)$ or $\beta = (x',y',t')$, then 
$|X^J S_P(\alpha,\beta)| \les |d(\alpha,\beta)|^{-|J|} |B_{CC}(\alpha,d(\alpha,\beta))|^{-1}$. This yields an immediate corollary.
\begin{cor}\label{cor:size of Szego}
Let $M_P$ be a model as in Theorem \ref{thm:P = ap} where $p$ is a convex polynomial. If $X^J$ is a multiindex of operators $X_1, X_2$ acting
in either $\alpha = (x,y,t)$ or $\beta = (x',y',t')$, then there exists a constant $C_{|J|}>0$ so that on $\supp \delta_0[(s-s') - b(t_n-t_n')]$
\[
|X^J S_M(\alpha,\beta)| \leq C_{|J|} \frac{|d(\alpha,\beta)|^{-|J|} } {|B_{CC}(\alpha,d(\alpha,\beta))|}.
\]
\end{cor}
The proof is immediate, given the fact that $X_1$ and $X_2$ are tangential on the manifold where $s-s' = b(t_n-t_n')$.

If $p_j(x) = a_jx^2$, then $M$ is an example of quadric submanifold. Quadrics have been studied extensively \cite{BoRa11,BoRa13q,ChTi00,CaChTi06,BeGaGr96,BeGaGr00,Gav77,Hul76}
and \cite{CaChMa09}, in particular for a more extensive background. In this case, we can compute all of the integrals explicitly and prove the following theorem.
\begin{thm}\label{thm:quadric} Let $M_P$ be the quadric submanifold defined by
\[
M_{a|x|^2} = \{(z,w)\in\C\times\C^n: \Imm w = x^2 a \}
\]
where $a = (a_1,\dots,a_n)\in\R^n$ and $a_n>0$. Then
\begin{align*}
S_{a|x|^2}\big((x,y,t),(x',y',t')\big) 
&= S_{a_n |x|^2}\big((x,y,t_n),(x',y',t_n')\big) \delta_0[a_n(s-s') - b(t_n-t_n')] \\
&= \frac{2 a_n \delta_0[a_n(s-s') - b(t_n-t_n')]}{\bigl(\pi a_n[(x-x')^{2} + (y-y')^{2}] - 2\pi i[(t_n-t_n')+a_n(x+x')(y-y')]\bigr)^{2}}
\end{align*}
\end{thm}

\begin{rem}The condition that $a_n>0$ is not essential -- we really require that $a_n\neq 0$, but we keep $a_n>0$ for simplicity. Also, in the proof of Theorem \ref{thm:quadric}
we explicitly compute the $S_{\lam x^2}((x,y,t),(x',y',t'))$ where $\lam>0$, and from that expression, we can see there is nothing distinguished about the $n$th coordinate, except the
fact that $a_n\neq 0$.
\end{rem}

The outline of the remainder of the paper consists of the proofs of the main theorems in Section \ref{sec:proofs} and a discussion of the control geometry in 
Section \ref{sec:control geometry}.

%
%
\section{Proofs of the Main Theorems}\label{sec:proofs}

\subsection{Proof of the Szeg\"o kernel formula}\label{subsec:proof of Szego formula}
\begin{proof}[Proof of Theorem \ref{thm:expression for S}]
The proof of Theorem \ref{thm:expression for S} follows from two observations. The first is that $\bar L$ is translation invariant in $y$ and $t$. This means we can take the partial
Fourier transform in $y$ and $t$. Given a function $f(x,y,t)$, we define the partial Fourier transform of $f$ to be
\[
\opF f(x,\eta,\tau) = \hat f(x,\eta,\tau) = \iint_{\R\times\R^n} e^{-2\pi i(y,t)\cdot(\eta,\tau)} f(x,y,t)\, dy\, dt.
\]
Under $\opF$, with $(\eta,\tau)$ as the transform variables of $(y,t)$, the operator
\[
\bar L \mapsto \hat{\bar L} = \frac{\p}{\p x} - 2\pi \eta + 2\pi P'(x)\cdot \tau = e^{2\pi(x\eta - P(x)\cdot\tau)} \frac{\p}{\p x} e^{-2\pi(x\eta - P(x)\cdot\tau)}.
\]
Set $\Psi(x,\eta,\tau) = e^{-2\pi(x\eta - P(x)\cdot\tau)}$ and $M_\Psi: L^2(\R,dx)\to L^2(\R,e^{4\pi(x\eta-P(x)\cdot\tau)})$ to be the isometry defined by 
$f \mapsto \Psi f$. Since $M_\Psi$ and $\opF$ are isometries, $\bar Lf =0$ if and only if $\frac{d}{dx}\{M_\Psi \opF f\}=0$.

The second observation is that $\ker \frac{d}{dx}$ are constant functions. The function $f=1$ is in $L^2(\R,e^{4\pi(x\eta-P(x)\cdot\tau)})$ exactly when
$C_{\eta,\tau}<\infty$. Assuming $C_{\eta,\tau}<\infty$, then the projection of $g$ onto $\ker \frac{d}{dx}$ is the operator $P_{\eta,\tau}$ given by
\[
P_{\eta,\tau}g(x) = P_{\eta,\tau}g = \frac{1}{C_{\eta,\tau}} \int_\R g(x') e^{4\pi(x'\eta - P(x)\cdot\tau)}\, dx'.
\]
If the operator $P = P_{\eta,\tau}$ on $L^2(\R,e^{4\pi(x\eta-P(x)\cdot\tau)})$ with the understanding that
$P_{\eta,\tau}=0$ when $(\eta,\tau)\not\in\Sigma$. Consequently,
\[
S = \opF^{-1} M_{\Psi^{-1}} P M_\Psi \opF.
\]
Expanding the right-hand side yields the desired formula.
\end{proof}

\subsection{Proof of the Szeg\"o kernel formula when $P = ap$}\label{subsec:proof of Szego when P=ap}
\begin{proof}[Proof of Theorem \ref{thm:P = ap}]
We use the following notation: $a = (a_1,\dots,a_n)$, $b=(a_1,\dots,a_{n-1})$, $\tau = (\sigma,\tau_n)$. Also, $a\cdot\tau = \tau_n+b\cdot\sigma$. 
Since $\lim_{|x|\to\infty}\frac{p(x)}{|x|}=\infty$, $C_{\eta,\tau}<\infty$ if and only if $a\cdot \tau >0$ (which is equivalent to $\tau_n > -b\cdot\sigma$), and this condition
is independent of $\eta$. We use a superscript to denote which model to which various expressions refer.
Also,
$P(x) = p(x) a$, so
\[
C_{\eta,\tau}^{ap} = \int_\R e^{4\pi(x\eta-p(x)a\cdot \tau)}\, dx = C_{\eta,\tau_n+\sigma\cdot b}^p.
\]
Consequently, we use Theorem \ref{thm:expression for S} and compute that
\begin{align*}
&S_{ap}\big((x,y,t),(x',y',t')\big)
= \int_{\Sigma_{ap}} \frac{1}{C_{\eta,\tau}^{ap}} e^{2\pi\eta(x+x'+i(y-y'))} e^{-2\pi \tau\cdot(a(p(x)+p(x')) - i(t-t'))}\, d\tau\, d\eta \\
&= \int_{\R^{n-1}} \int_\R  \int_{\tau_n = -b\cdot\sigma}^\infty \frac{1}{C_{\eta,\tau_n+\sigma\cdot b}^p} e^{2\pi\eta(x+x'+i(y-y'))}
e^{-2\pi(\tau_n+b\cdot\sigma)(p(x)+p(x')-i(t_n-t_n'))} e^{2\pi i \sigma\cdot[(s-s')-b(t_n-t_n')]}\, d\tau_n\,d\eta \,d\sigma
\end{align*}
where the last line uses the equality $\tau\cdot(t-t') = \sigma\cdot(s-s') + \tau_n(t_n-t_n')$ and the fact that $a_n=1$.
Shifting the variable $\tau+b\cdot\sigma \mapsto \tau$, comparing the resulting formula to (\ref{eqn:1-d Szego}),
and recognizing that resulting integration in $\sigma$ results in a $\delta_0[(s-s') - b(t_n-t_n')]$ finishes the proof.
\end{proof}

\subsection{The quadric case}\label{subsec:quadric}
\begin{proof}[Proof of Theorem \ref{thm:quadric}]
We use Theorem \ref{thm:P = ap}. In the case that $p(x)=x^2$ and $n=1$, we see that if $\lambda>0$, then $C^{\lambda x^2}_{\eta,\tau} = \frac{e^{\frac{\pi\eta^2}{\lambda \tau}}}{2\sqrt{\lambda\tau}}$, 
so applying Nagel's formula yields
\begin{align*}
S_{\lambda x^2}((x,y,t),(x',y',t')) &= 
 \int_{0}^{\infty}  \int_{\mathbb{R}}  2\sqrt{\tau\lambda} e^{-\frac{\pi\eta^2}{\lambda\tau}} e^{-2\pi\tau[\lambda(x^2 + x'^2)-i(t-t')]}
         e^{2\pi\eta[(x+x') + i(y-y')]} \,d\eta\,d\tau \\
             &= \frac{2 \lam}{\bigl(\pi\lam[(x-x')^{2} + (y-y')^{2}] - 2\pi i[(t-t')+\lam(x+x')(y-y')]\bigr)^{2}}
\end{align*}
We are not assuming that $a_n=1$, and we could use a change of variables to reduce to this case, but it is simpler to make the change of variables
$\tilde \sigma = \sigma$, and $\tilde \tau_n = a\cdot \tau$.  Next, an easy computation establishes that $\Sigma_{ax^2} = \{(\eta,\tau) : a\cdot\tau>0\} = \{(\eta,\tilde\tau):\tilde\tau_n>0\}$ and
$C_{\eta,\tau}^{ax^2} = \frac{e^{\frac{\pi\eta^2}{a\cdot\tau}}}{2\sqrt{a\cdot\tau}}$
Consequently, we compute that $\tau_n = (\tilde \tau_n - \sigma\cdot b)/a_n$ and
\begin{align*}
&S_{ax^2}\big((x,y,t),(x',y',t')\big) = \int_{\Sigma_{ax^2}} 2\sqrt{a\cdot\tau} e^{-\frac{\pi\eta^2}{a\cdot\tau}} e^{2\pi\eta[(x+x')+i(y-y')]} e^{-2\pi \tau \cdot[a(x^2+x'^2) - i(t-t')]}\, d\eta\, d\tau \\
&= \int_\R \int_0^\infty \int_{\R^{n-1}} 2\sqrt{\tilde\tau_n} e^{-\frac{\pi\eta^2}{\tilde\tau_n}} e^{2\pi\eta[(x+x')+i(y-y')]}e^{-2\pi\tilde\tau_n(x^2+x'^2)} e^{2\pi i \sigma\cdot(s-s')}
e^{2\pi i(\tilde\tau_n-\sigma\cdot b)(t_n-t_n')/a_n}\, d\sigma\,d\tilde\tau_n\,d\eta\\ 
&= \frac{1}{a_n} \int_\R \int_0^\infty\int_{\R^{n-1}} 2\sqrt{\tilde\tau_n} e^{-\frac{\pi\eta^2}{\tilde\tau_n}} e^{2\pi\eta[(x+x')+i(y-y')]} 
e^{-2\pi\tilde\tau_n[(x^2+x'^2)-i(t_n-t_n')/a_n]} e^{2\pi i \sigma\cdot[(s-s')-\frac b{a_n}(t_n-t_n')]}\, d\sigma\,d\eta\,d\tilde\tau_n\\
&= \frac{1}{a_n}  \frac{2 \delta_0[a_n(s-s') - b(t_n-t_n')]}{\bigl(\pi[(x-x')^{2} + (y-y')^{2}] + 2\pi i[(t_n-t_n')/a_n+(x+x')(y-y')]\bigr)^{2}} \\
&=  S_{a_n |x|^2}\big((x,y,t_n),(x',y',t_n')\big) \delta_0[a_n(s-s') - b(t_n-t_n')].
\end{align*}
\end{proof}

%
%
\section{Connection to the control geometry}\label{sec:control geometry}

Since on a finite type domain boundary in $\mathbb{C}^2$ the Szeg\"o kernel is governed by the control metric, 
we may naturally  ask whether the same holds on a codimension CR manifold with at least two totally real directions. 
It is easy to extend the notion of finite commutator
type; we simply require that the real and imaginary parts of $\bar{L}$, $X_{1}$ and $X_{2}$, along with a 
finite number $m$ of their iterated commutators, span the real tangent space at every point of $M_{P}$.
\begin{defn}
	\label{cmetric2}
	With notation as above, let $\{Y_{1},\ldots ,Y_{q}\}$ be some enumeration
	of the vector fields $X_{1}$, $X_{2}$, and all their iterated commutators of length less than or equal $m$. 
	Define the ``degree'' of each vector field $Y_{j}$ by
	$$\textrm{d}(Y_{j}) = \textrm{\ length of the iterated commutator that forms } Y_{j}$$
	Now let the distance between $p, q \in M_P$ be the infimum of $\delta > 0$
	such that there is an absolutely continuous map $\gamma: [0,1] \mapsto M_P$
	with $\gamma(0) = p$, $\gamma(1) = q$ so that for almost all $r \in (0,1)$
	$$\gamma'(r) = \sum_{j=1}^{q} c_{j}(t)Y_{j}\left(\gamma(r)\right), \qquad |c_{j}(r)| < \delta^{\textrm{d}(Y_{j})}$$
\end{defn}

Under this condition, the control distance yields a metric, but one which currently defies any tractable description. 
(Indeed, even on a domain boundary $M_p$, a serious amount of work is required to prove the 
equivalence of the control metric to the pseudometrics investigated by Nagel et al. \cite{NaStWa85}.) 
Although $M_{ap(x)}$ is not of finite type, there is a submanifold of $M_{ap(x)}$ on which the control distance is finite and a direct
connection to the Szeg\"o\ kernel on $M_{ap(x)}$.
Here $X_{1} = \pure{x}$ and $X_{2} = \pure{y} - p'(x)a\cdot\frac{\p}{\p t}$, 
so every potentially non-zero commutator is of the form
$$Y_{k} = [\overbrace{X_{1},[X_{1},\cdots,[X_{1}}^{\text{$k-1$ times}}, X_{2}]\cdots]]\qquad (2 \leq k \leq m)$$
That is,
$$Y_{k} = -p^{(k)}(x) a\cdot \frac{\p}{\p t}.$$
This forces $\{X_{1},X_{2},Y_{k_{1}}, Y_{k_{2}}\}$
to span only the subspace generated by $\{\frac{\p}{\p x}, \frac{\p}{\p y}, a \cdot \frac{\p}{\p t}\}$. Thus
 the real tangent space is never spanned, at any 
point of $M_{ap(x)}$, and $\{X_{1}, X_{2}\}$ do not generate a finite control metric.

We may still consider control distance on 
$M_{ap(x)}$. This distance is finite and less than some $\delta > 0$ 
if and only if there exists an absolutely continuous curve 
$\gamma: [0,1] \mapsto M_{ap(x)}$ such that 
 $$\gamma'(\mvs) = c_{0}(\mvs)X_{1}\left(\gamma(\mvs)\right)
 	+ c_{1}(\mvs)X_{2}\left(\gamma(\mvs)\right)
	+ \sum_{k=2}^{m}c_{k}(\mvs)Y_{k}\left(\gamma(\mvs)\right)$$
with $|c_{0}(\mvs)|, |c_{1}(\mvs)| < \delta$ and $|a_{k}(\mvs)| < \delta^{k}$ for
almost all $\mvs \in (0,1)$. Given our previous comments, such a curve 
$\gamma(\mvs) = (\gamma_{1}(\mvs), \gamma_{2}(\mvs), \gamma_{3}(\mvs), \dots, \gamma_{n+2}(\mvs))$
can only exist if $(\gamma_{3}',\dots,\gamma_{n+2}')$ is parallel to $a$ for almost all $\mvs \in (0,1)$.
From Theorem \ref{thm:P = ap}, the Szeg\"o\ kernel on $M_{ap(x)}$
is a singular distribution supported on exactly the subspace where the control distance on 
$M_{ap(x)}$ is finite. On this subspace, the control ball is well-defined and exactly 
determines the size of the Szeg\"o kernel on $M_{ap(x)}$, treating the subspace as an $\R^3$ and applying the Nagel et. al. machinery.

\bibliographystyle{alpha}
\bibliography{mybib8-13-14}

\end{document}